\documentclass[12pt]{amsart}

\usepackage{amsmath,amssymb,amsfonts,amstext,amsthm,pslatex,enumerate}

\newtheorem{theorem}{Theorem}[section]
\newtheorem{proposition}[theorem]{Proposition}
\newtheorem{lemma}[theorem]{Lemma}

\theoremstyle{definition}
\newtheorem{definition}[theorem]{Definition}

\theoremstyle{remark}
\newtheorem{remark}[theorem]{Remark}

\numberwithin{equation}{section}

\begin{document}

\title[MUH and Schemes]
{Mutually Unbiased Bush-type Hadamard Matrices and Association Schemes}
\renewcommand{\thefootnote}{\fnsymbol{footnote}}
\author{Hadi Kharaghani}
\address{Department of Mathematics and Computer Science, University of Lethbridge,
Lethbridge, Alberta, T1K 3M4, Canada}
\email{kharaghani@uleth.ca}
\author{Sara Sasani}
\address{Department of Mathematics and Computer Science, University of Lethbridge,
Lethbridge, Alberta, T1K 3M4, Canada}
\email{sasani@uleth.ca}
\author{Sho Suda}
\address{Department of Mathematics Education, Aichi University of Education, Kariya, Aichi, 448-8542, Japan}
\email{suda@auecc.aichi-edu.ac.jp}

\begin{abstract}

It was shown by LeCompte, Martin, and Oweans in 2010 that the existence of mutually unbiased Hadamard matrices and the identity matrix, which coincide with mutually unbiased bases, is equivalent to that of a $Q$-polynomial association scheme of class four which is both $Q$-antipodal and $Q$-bipartite. 
We prove that the existence of a set of mutually unbiased Bush-type Hadamard matrices is equivalent to that of an association scheme of class five.
As an application of this equivalence, we obtain the upper bound of the number of mutually unbiased Bush-type Hadamard matrices of order $4n^2$ to be $2n-1$. This is in contrast to the fact that the upper bound of mutually unbiased Hadamard matrices of order $4n^2$ is $2n^2$. 
We also discuss a relation of our scheme to some fusion schemes which are $Q$-antipodal and $Q$-bipartite $Q$-polynomial of class $4$.

\end{abstract}

\maketitle
\footnotetext[1]{Hadi Kharaghani is supported by an NSERC Discovery Grant.}
\vskip5mm

\section{Introduction}
A {\it  Hadamard} matrix is a matrix $H$ of order $n$ with
entries in $(-1,1)$ and orthogonal rows in the usual
inner product on ${\mathbb{R}}^n$.
Two Hadamard matrices $H$ and $K$ of order $n$ are called {\it unbiased}
if $HK^{t}={\sqrt n} L$ for some Hadamard matrix $L$,
where $K^t$ denotes the transpose of $K$.
In this case, it follows that $n$ must be a perfect square.
A Hadamard matrix of order $n$ for which the row sums and column sums
are all the same, necessarily $\sqrt n$, is called {\it regular},
see \cite{ys}.
\begin{definition}
A \emph{Bush-type} Hadamard matrix
is a  block matrix $H=[H_{ij}]$ of
order $4n^2$ with block size $2n$,
$H_{ii}=J_{2n}$
and $H_{ij}J_{2n}
= J_{2n}H_{ij}=0$, $i\ne j$, $1\le i \le 2n$, $1\le j\le 2n$
where $J_{2n}$ is the $2n$ by $2n$ matrix of all 1 entries.
\end{definition}
It is known that for odd values of $n$ there is no pair of  unbiased Bush-type Hadamard matrices of order $4n^2$ \cite{bk}.
However, on the contrary
for even values of $n$, there are unbiased Bush-type Hadamard matrices of order $4n^2$ \cite{hko}. 
(In \cite[Theorem 13]{hko}, an assumption is needed. The modified version will be presented in Section~\ref{5classas}.) 
One very important property
of unbiased Bush-type Hadamard matrices,  as is shown in section 3, is the fact that they are closed under the property of being of Bush-type. 

It was shown by LeCompte, Martin, and Oweans in \cite{lmo} that the existence of mutually unbiased Hadamard matrices and the identity matrix, which coincide with mutually unbiased bases, is equivalent to that of a $Q$-polynomial association scheme of class four which is both $Q$-antipodal and $Q$-bipartite. 

Our aim in this
paper is to show that the existence of unbiased Bush-type Hadamard matrices is equivalent to the existence
of a certain association scheme of class five. 
As an application of this equivalence, we obtain the upper bound of the number of mutually unbiased Bush-type Hadamard matrices of order $4n^2$ is $2n-1$, whereas the upper bound of mutually unbiased Hadamard matrices of order $4n^2$ is $2n^2$ \cite[Theorem 2]{hko}. 
Also we discuss a relation of our scheme to such association schemes of class four. 

\section{Association schemes}
A \emph{symmetric $d$-class association scheme}, see \cite{bi}, 
with vertex set $X$ of size $n$ and $d$ classes
is a set of symmetric $(0,1)$-matrices $A_0, \ldots, A_d$, which are not equal to zero matrix, with
rows and columns indexed by $X$, such that:
\begin{enumerate}
\item $A_0=I_n$, the identity matrix of order $n$.
\item $\sum_{i=0}^d A_i = J_n$, the matrix of order $n$ with all one's entries.
\item For all $i$, $j$, $A_iA_j=\sum_{k=0}^d p_{ij}^k A_k$
for some $p_{ij}^k$'s.
\end{enumerate}
It follows from property (3) that the $A_i$'s necessarily commute.
The vector space spanned by $A_i$'s forms a commutative algebra, denoted by $\mathcal{A}$ and called the \emph{Bose-Mesner algebra} or \emph{adjacency algebra}.
There exists a basis of $\mathcal{A}$ consisting of primitive idempotents, say $E_0=(1/n)J_n,E_1,\ldots,E_d$. 
Since  $\{A_0,A_1,\ldots,A_d\}$ and $\{E_0,E_1,\ldots,E_d\}$ are two bases in $\mathcal{A}$, there exist the change-of-bases matrices $P=(P_{ij})_{i,j=0}^d$, $Q=(Q_{ij})_{i,j=0}^d$ so that
\begin{align*}
A_j=\sum_{i=0}^d P_{ij}E_j,\quad E_j=\frac{1}{n}\sum_{i=0}^d Q_{ij}A_j.
\end{align*}   
Since disjoint $(0,1)$-matrices $A_i$'s form a basis of $\mathcal{A}$, the algebra $\mathcal{A}$ is closed under the entrywise multiplication denoted by $\circ$. 
The \emph{Krein parameters} $q_{ij}^k$'s are defined by $E_i\circ E_j=\frac{1}{n}\sum_{k=0}^dq_{ij}^k E_k$. 
The \emph{Krein matrix} $B_i^*$ is defined as $B_i^*=(q_{ij}^k)_{j,k=0}^d$.

Each of the matrices $A_i$'s can be considered as the adjacency matrix of some graph without multiedges. The scheme is \emph{imprimitive} if,
on viewing the $A_i$'s as adjacency matrices
of graphs $G_i$ on vertex set $X$, at least one of the $G_i$'s, $i \ne 0$, is disconnected. 
Then there exists a set $\mathcal{I}$ of indices such that $0$ and such $i$ are elements of $\mathcal{I}$ and $\sum_{j\in\mathcal{I}}A_j=I_p\otimes J_q$ for some $p,q$ with $p<n$. 
Thus the set of $n$ vertices $X$ are partitioned into $p$ subsets called \emph{fibers}, each of which has size $q$. 
The set $\mathcal{I}$ defines an equivalence relation on $\{0,1,\ldots,d\}$ by $j\sim k$ if and only if $p_{ij}^k\neq 0$ for some $i\in \mathcal{I}$.  
Let $\mathcal{I}_0=\mathcal{I},\mathcal{I}_1,\ldots,\mathcal{I}_t$ be the equivalent classes on $\{0,1,\ldots,d\}$ by $\sim$. 
Then by \cite[Theorem~9.4]{bi} there exist $(0,1)$-matrices $\overline{A}_j$ ($0\leq j\leq t$) such that 
\begin{align*}
\sum_{i\in \mathcal{I}_j}A_i=\overline{A}_j\otimes J_q,
\end{align*}
and the matrices $\overline{A}_j$ ($0\leq j\leq t$) define an association scheme on the set of fibers. 
This is called the \emph{quotient association scheme} with respect to $\mathcal{I}$

For fibers $U$ and $V$, let $\mathcal{I}(U,V)$ denote the set of indices of adjacency matrices that has an edge between $U$ and $V$. 
We define a $(0,1)$-matrix $A_i^{UV}$ by 
\begin{align*}
(A_i^{UV})_{xy}=\begin{cases}
1 &\text{ if } (A_i)_{xy}=1, x\in U, y\in V,\\
0 &\text{ otherwise}.
\end{cases}
\end{align*}
\begin{definition}
An imprimitive association scheme is called uniform if its quotient association scheme is class 1 and there exist integers $a_{ij}^k$ such that for  all fibers $U,V,W$ and $i\in\mathcal{I}(U,V),j\in\mathcal{I}(V,W)$, we have 
\begin{align*}
A_i^{UV}A_i^{VW}=\sum_{k}a_{ij}^k A_k^{UW}.
\end{align*}
\end{definition}

\section{Class 5 Association Scheme}\label{5classas}
Let $\{ H_1, H_2, \ldots, H_m\}$ be a set
of Mutually Unbiased Regular Hadamard (MURH) matrices of order $4n^2$ with $m\ge 2$.
Let
\begin{align}M =
\left[
\begin{array}{c}
I\\H_1/2n\\H_2/2n\\\ldots\\H_m/2n
\end{array}
\right]
\left[
\begin{array}{ccccc}
I&H_1^t/2n&H_2^t/2n&\ldots&H_m^t/2n
\end{array}
\right].\label{eq:M}
\end{align}
be the {\em Gramian} of the set of matrices $\{I, \frac{1}{2n}H_1, \frac{1}{2n}H_2, \ldots, \frac{1}{2n}H_m\}$.
Let $B=2n(M-I)$.
Then $B$ is
a symmetric $(0,-1,1)$-matrix. 
Let  $$B=B_1-B_2,$$ where
$B_1$ and $B_2$ are  disjoint 
$(0,1)$-matrices.  By reworking a result of Mathon, see \cite{d},
 we have the following:
\begin{lemma}\label{3-assoc}
Let $I=I_{4n^2(m+1)}$, $B_1, B_2$ and $B_3=I_{m+1}\otimes J_{4n^2} - I_{4n^2(m+1)}$. Then, $I,B_1,B_2,B_3$
form a 3-class association scheme.
\end{lemma}
\begin{proof}
The intersection numbers are:
\begin{align}
B_1^2&=
(2n^2+ n)mI+(n^2+\frac{3}{2}n)(m-1)B_1\notag\break\displaybreak[0]\\
&+(n^2+n/2)(m-1)B_2+(n^2+n)B_3, \notag\break\displaybreak[0]\\
B_2^2&=
(2n^2- n)mI+(n^2-\frac{1}{2}n)(m-1)B_1\notag\\ &+(n^2-\frac{3}{2}n)(m-1)B_2+(n^2-n)B_3,\notag\break\displaybreak[0]\\
B_1B_2 &=(n^2-n/2)(m-1)B_1\notag\break\displaybreak[0]\\
&+(n^2+n/2)(m-1)B_2+n^2mB_3,\notag\break\displaybreak[0]\\
B_1B_3 &=(2n^2+n-1)B_1+(2n^2+n)B_2,\notag\break\displaybreak[0]\\
B_2B_3 &=
(2n^2-n)B_1+(2n^2-n-1)B_2.\notag \qedhere
\end{align}
\end{proof}

We now impose a further structure on the regular Hadamard matrices $H_i$s and assume that 
they are all of Bush-type. First we need the following.

\begin{lemma}\label{lem:bh}
Let $H$ and $K$ be two unbiased Bush-type Hadamard matrices of order $4n^2$. 
Let $L$ be a $(1,-1)$-matrix so that $HK^t=2nL$. 
Then $L$ is a Bush-type Hadamard matrix. 
\end{lemma}
\begin{proof}
Let $X=I_{2n}\otimes J_{2n}$, then $L$ is of Bush-type if and only if  $LX=XL=2nX$.
We calculate $LX$. 
$$LX=\frac{1}{2n}HK^t X=2nX.$$
Similarly, we have $XL=2nX$. 
Thus $L$ is a Bush-type Hadamard matrix. 
\end{proof}

This enables us to add two more classes and we have
the following.

\begin{theorem}\label{mn5}
Let  $B_1, B_2$ denote the matrices defined above.
Let 
\begin{itemize}
  \item $A_0 = I_{4n^2(m+1)}$
  \item $A_1 = I_{m+1} \otimes I_{2n} \otimes (J_{2n} - I_{2n})$
  \item $A_2 = I_{m+1} \otimes (J_{2n} - I_{2n}) \otimes J_{2n}$
\item $A_3 = (J_{m+1} - I_{m+1} ) \otimes I_{2n} \otimes J_{2n}$
\item $A_4 = B_1 - A_3$ 
\item $A_5 = B_2$
\end{itemize}
Then, $A_0,A_1,A_2,A_3,A_4,A_5$ form a 5-class association scheme.
\end{theorem}

\begin{proof}
We work out the intersection numbers, using some of the relations in Lemma \ref{3-assoc}. Note that $A_0+A_1$, $A_3$, ($A_0+A_1+A_2$) are block matrices 
of block size $2n$, ($4n^2$, respectively), where each block
is either the zero or the all one's matrix. On the other hand $A_4$ and $A_5$ are block matrices of block size $2n$, where 
the blocks are either the zero matrix or of row and column sum $n$. So, it is straightforward computation to see the following:
\begin{align}
A_1A_1 &= (2n-1)A_0+(2n-2)A_1.\notag\break\displaybreak[0]\\
A_1A_2 &= (2n-1)A_2.\notag\break\displaybreak[0]\\
A_1A_3 &= (2n-1) A_3.\notag\break\displaybreak[0]\\
A_1A_4 &= (n-1)A_4 + nA_5.\notag\break\displaybreak[0]\\
A_1A_5 &= nA_4 +(n-1)A_5.\notag\break\displaybreak[0]\\
A_2A_2 &=  2n(2n-1)A_0+ 2n(2n-1)A_1 + 2n(2n-2) A_2. \notag\break\displaybreak[0]\\
A_2A_3 &= 2n(A_4+A_5).\notag\break\displaybreak[0]\\
A_2A_4 &=(2n-1)nA_3 + (2n-2)n(A_4+A_5). \notag\break\displaybreak[0]\\
A_2A_5 &=(2n-1)nA_3 + (2n-2)n(A_4+A_5).\notag\break\displaybreak[0]\\
A_3A_3 &= 2mn(A_0+A_1) + 2n(m-1)A_3.\notag\break\displaybreak[0]\\
A_3A_4 &= mnA_2 + (m-1)n(A_4+A_5).\notag\break\displaybreak[0]\\
A_3A_5 &=  mnA_2 + (m-1)n(A_4+A_5).\notag
\end{align}
Using these, the facts that $A_3+A_4=B_1$, $A_5(A_3+A_4)=B_2B_1$, and the intersection numbers in Lemma \ref{3-assoc} we have:
\begin{align}
A_4A_5 &= n^2mA_1 + m(n^2-n) A_2 +(n^2 - \frac{n}{2})(m-1)A_3 \notag\\
&+ (n^2 - \frac{3n}{2})(m-1)A_4 + (n^2 -\frac{n}{2} )(m-1)A_5.\notag
\end{align}
Finally, noting that $A_4-A_5$ is a block matrix of block size $2n$, where 
the blocks are either the zero matrix or of row and column sum zero, it follows that
 $$(A_4+A_5)(A_4-A_5)=0,$$ so we have
 \begin{align}
A_4A_4 &= A_5A_5 =(2n^2 - n)m I  + (n^2-n)m(A_1+A_2) +\notag\\
&(n^2 - \frac{n}{2})(m-1) ( A_3 + A_4) + (n^2 - \frac{3n}{2})(m-1) A_5.\notag
\end{align}
\end{proof}

The existence and a construction method for MUBH  matrices to use mutually suitable Latin squares were given in \cite[Theorem 13]{hko}.
However, in order to obtain Bush-type Hadamard matrices as defined here, an additional assumption on the MSLS is needed as follows.
\begin{proposition}
If there are $m$ mutually suitable Latin squares of size $2n$ with all one entries on diagonal and a Hadamard matrix of order $2n$, then there are $m$ mutually unbiased Bush-type Hadamard matrices of order $4n^2$.
\end{proposition}
The construction is exactly same as \cite[Theorem 13]{hko}. 
The resulting mutually unbiased Hadamard matrices are all of Bush-type.
Indeed, each Latin square has the entries $1$ on diagonal, thus the resulting Hadamard matrix has the all ones matrices on diagonal blocks. 

The equivalence of MOLS and MSLS was given in \cite[Lemma 9]{hko}. 
The assumption on MOLS corresponding to MSLS with all ones entries on diagonal is that each Latin square has $(1,2,\cdots,n)$ as the first row. 
The MOLS having this property is constructed by the use of finite fields as follows. 
For each $\alpha\in \mathbb{F}_q\setminus\{0\}$, define $L_{\alpha}$ as $(i,j)$-entry equal to $\alpha i+j$, where $i,j\in\mathbb{F}_q$. 
By switching rows so that the first row corresponds to $0\in \mathbb{F}_q$ and mapping $\mathbb{F}_q$ to $\{1,2,\ldots,n\}$ such that each first row becomes $(1,2,\ldots,n)$, we obtain the desired MOLS. 
Thus we have the same conclusion as \cite[Corollary~15]{hko}.
 
\begin{remark}\label{rem:1}
\begin{enumerate}
 \item Rewriting $A_4A_4=A_5A_5$ as:
\begin{align}
A_4A_4 &= A_5A_5= n^2m I +(n^2-n)mJ + n(\frac{m+1}{2}-n)(A_3+A_4)\notag\\
& + n(\frac{3}{2}-\frac{m}{2}-n) A_5.\notag
\end{align}
It is seen that, for $m=2n-1$, $A_5$ is the adjacency matrix of a strongly regular graph and $A_4$
is the adjacency matrix of a Deza Graph, see \cite{dd,efhhh}. This is true for $n=2^k$, for each integer $k\ge 1$. 

\item The association scheme of class $5$ is uniform.  
Any two fibers define a coherent configuration, which is a strongly regular design of the second kind, see \cite{h}.  
The first and second eigenmatrices and $B_5^*$ are as follows:
\begin{align*}
P&=\begin{pmatrix}
1&2n-1&2n(2n-1)&2nm&n(2n-1)m&n(2n-1)m\\
1&-1&0&0&nm&-nm\\
1&2n-1&-2n&2nm&-nm&-nm\\
1&2n-1&-2n&-2n&n&n\\
1&-1&0&0&-n&n\\
1&2n-1&2n(2n-1)&-2n&-n(2n-1)&-n(2n-1)
 \end{pmatrix},\displaybreak[0]\\
Q&=\begin{pmatrix}
1&2n(2n-1)&2n-1&(2n-1)m&2n(2n-1)m&m\\
1&-2n&2n-1&(2n-1)m&-2nm&m\\
1&0&-1&-m&0&m\\
1&0&2n-1&-2n+1&0&-1\\
1&2n&-1&1&-2n&-1\\
1&-2n&-1&1&2n&-1
 \end{pmatrix},\displaybreak[0]\\
B_5^*&=\begin{pmatrix}
0&0&0&0&0&1\\
0&0&0&0&1&0\\
0&0&0&1&0&0\\
0&0&m&m-1&0&0\\
0&m&0&0&m-1&0\\
m&0&0&0&0&m-1
\end{pmatrix}.
\end{align*}
Thus the association scheme certainly satisfies \cite[Proposition 4.7]{dmm}. 
Since the Krein number $q_{1,2}^1=\frac{2n-m-1}{m+1}$ must be positive, 
we obtain $m\leq 2n-1$ holds. 
This means that the number of MUBH matrices of order $4n^2$ is at most $2n-1$. 
The example attaining the upper bound is given in \cite[Corollary 15]{hko}.
\item 
The first, second eigenmatrices and the Krein matrix $B_1^*$ of the class $3$ association scheme are as follows:
\begin{align*}
P&=\begin{pmatrix}
1&n(2n+1)m&n(2n-1)m&4n^2-1\\
1&nm&-nm&-1\\
1&-n&n&-1\\
1&-n(2n+1)&-n(2n-1)&4n^2-1
 \end{pmatrix},\displaybreak[0]\\
Q&=\begin{pmatrix}
1&4n^2-1&(4n^2-1)m&m\\
1&2n-1 &-2n+1&1\\
1&-2n-1&2n+1&1\\
1&-1&-m&m
 \end{pmatrix},\displaybreak[0]\\
B_1^*&=\begin{pmatrix}
0&1&0&0\\
4n^2-1&\frac{2(2n^2-m-1)}{m+1}&\frac{4n^2}{m+1}&0\\
0&\frac{4n^2 m}{m+1}&\frac{(4n^2-2)m-2}{m+1}&4n^2-1\\
0&0&1&0
\end{pmatrix}.
\end{align*}
This association scheme is a $Q$-antipodal $Q$-polynomial scheme of class $3$.  
By \cite[Theorem 5.8]{d}, this scheme comes from a linked systems of symmetric designs.
\end{enumerate}
\end{remark}

Next we show the converse implication as follows.
\begin{theorem}
Assume that there exists an association scheme with the same eigenmatrices in Remark~\ref{rem:1}. 
Then there exists a set of MUBH $\{H_1,\ldots.H_m\}$ of order $4n^2$.  
\end{theorem}
\begin{proof}
Let $A_0,\ldots,A_5$ be the adjacency matrices of an association scheme with the same eigenmatrices in Remark~\ref{rem:1}. 
Let $B_0=A_0$, $B_1=A_3+A_4$, $B_2=A_5$ and $B_3=A_1+A_2$.
By Remark~\ref{rem:1} $B_i$ 's form a linked system of symmetric designs. 
Thus we rearrange the vertices so that $B_3=I_{m+1}\otimes J_{4n^2}-I_{4n^2(m+1)}$ . 
 
We first determine the form of $A_3$. 
Since $A_1$ is the adjacency matrix of an imprimitive strongly regular graph with eigenvalues $2n-1,-1$ with multiplicities $2n(m+1),2n(2n-1)(m+1)$,  
$A_1$ is $I_{2n(m+1)}\otimes (J_{2n}-I_{2n})$ after rearranging the vertices. 
By $B_3=I_{m+1}\otimes J_{4n^2}-I_{4n^2(m+1)}=A_1+A_2$, $A_2$ has the desired form. 
Since $B_3$ and $A_3$ are disjoint and $A_2A_3=2n(A_4+A_5)$, we obtain $A_3=(J_{m+1} - I_{m+1}) \otimes I_{2n} \otimes J_{2n}$. 

Letting $G=(m+1)(E_0+E_1+E_2)$, we have 
\begin{align*}
G&=(m+1)(E_0+E_1+E_2)\\
&=\frac{1}{4n^2}\sum_{i=0}^5(P_{0,i}+P_{1,i}+P_{2,i})A_i\\
&=A_0+\frac{1}{2n}A_3+\frac{1}{2n}A_4-\frac{1}{2n}A_5.
\end{align*}
Since $A_3+A_4+A_5=(J_{m+1}-I_{m+1})\otimes J_{2n} \otimes J_{2n}$, $G$ is the following form
\begin{align*}
G=\begin{pmatrix}
I_{2n}&\frac{1}{2n}H_{1,2}&\ldots&\frac{1}{2n}H_{1,m+1}\\
\frac{1}{2n}H_{2,1}&I_{2n}&\ldots&\frac{1}{2n}H_{2,m+1}\\
\vdots&\vdots&\ddots&\vdots\\
\frac{1}{2n}H_{m+1,1}&\frac{1}{2n}H_{m+1,2}&\ldots&I_{2n}
\end{pmatrix}
\end{align*}
where $H_{i,j}$ ($i\neq j$) is a $(1,-1)$-matrix.

We claim that  $H_k:=H_{k+1,1}$ ($1\leq k\leq m$) are mutually unbiased Bush-type Hadamard matrices. 
Let $\bar{A}$ denote the matrix obtained by restricted to the indices on the first and $(k+1)$-st blocks. 
We consider the principal submatrix $\bar{G}$.  
Since the association scheme is uniform, 
the restricting to indices on the first and second blocks yields an association scheme with the eigenmatrix $\bar{P}=(\bar{P_{ij}})_{i,j=0}^5$ obtained by putting $m=1$. 

Since $\bar{G}=(m+1)(\bar{E_0}+\bar{E_1}+\bar{E_2})$ holds and $\frac{m+1}{2}\bar{E_i}$ ($i=0,1,2$) are primitive idempotents of  the subscheme, we have $\bar{G}^2=2\bar{G}$. 
Expanding  the left hand-side to use the form $\bar{G}=\left( \begin{smallmatrix}I_{2n}&\frac{1}{2n}H_{k}^t\\ \frac{1}{2n}H_{k}&I_{2n} \end{smallmatrix}\right)$, we obtain 
\begin{align*}
\begin{pmatrix}
I_{2n}+\frac{1}{4n^2}H_{k}^tH_{k}&\frac{1}{n}H_{k}^t\\
\frac{1}{n}H_{k}& I_{2n}+\frac{1}{4n^2}H_{k}H_{k}^t
\end{pmatrix}=
\begin{pmatrix}
2I_{2n}&\frac{1}{n}H_{k}^t\\ \frac{1}{n}H_{k}&2I_{2n}
\end{pmatrix}. 
\end{align*}
This implies that $H_{k}$ is a Hadamard matrix of order $4n^2$.

Next we show $H_k$ is of Bush-type.
Now we calculate $\bar{A_3}\bar{G}$ in two ways. First we have
\begin{align*}
\bar{A_3}\bar{G}&=(m+1)\bar{A_3}(\bar{E_0}+\bar{E_1}+\bar{E_2})\\ 
&=(m+1)(\sum_{i=0}^5\bar{P_{i3}}\bar{E_i})(\bar{E_0}+\bar{E_1}+\bar{E_2})\\ 
&=(m+1)(\sum_{i=0}^2\bar{P_{i3}}\bar{E_i})\\
&=2n(m+1)(\bar{E_0}+\bar{E_2})\\ 
&=2n(m+1)(\frac{1}{4n^2(m+1)}\sum_{i=0}^5(\bar{Q_{i,0}}+\bar{Q_{i,2}})\bar{A_i})\\ 
&=(A_0+A_1+A_3)\\
&=\begin{pmatrix}
I_{2n} \otimes J_{2n}&I_{2n} \otimes J_{2n}\\
I_{2n} \otimes J_{2n}&I_{2n} \otimes J_{2n}
\end{pmatrix}.
\end{align*}
Second we have 
\begin{align*}
\bar{A_3}\bar{G}&=\begin{pmatrix}
0&I_{2n} \otimes J_{2n}\\
I_{2n} \otimes J_{2n}&0
\end{pmatrix}\begin{pmatrix}
I_{2n}&\frac{1}{2n}H_{k}^t\\
\frac{1}{2n}H_{k}&I_{2n}
\end{pmatrix}\\
&=\begin{pmatrix}
\frac{1}{2n}(I_{2n} \otimes J_{2n})H_{k}&I_{2n} \otimes J_{2n}\\
I_{2n} \otimes J_{2n}& \frac{1}{2n}(I_{2n} \otimes J_{2n})H_{k}^t
\end{pmatrix}.
\end{align*}
Comparing these two equations yields 
\begin{align*}
(I_{2n} \otimes J_{2n})H_{k}=(I_{2n} \otimes J_{2n})H_{k}^t=2nI_{2n} \otimes J_{2n}.
\end{align*}
This implies that $H_{k}$ is of Bush-type by Lemma~\ref{lem:bh}.
 
Finally we show $H_1,\ldots,H_m$ are unbiased. 
Let $k,k'$ be integers such that $1\leq k<k'\leq m$.
From now on, the overbar of matrices means the matrix obtained by restricted to the indices on the first, $(k+1)$-th, and $(k'+1)$-th blocks. 
We then have $\bar{G}^2=3\bar{G}$. 
Comparing the $(2,3)$-block, we obtain  
\begin{align*}
\frac{1}{4n^2}H_k H_{k'}^t+I_{2n}H_{k+1,k'+1}+H_{k+1,k'+1}I_{2n}=3H_{k+1,k'+1},
\end{align*}  
namely $\frac{1}{4n^2}H_k H_{k'}^t=H_{k+1,k'+1,}$. 
Since $H_{k+1,k'+1}$ is a $(-1,1)$-matrix,  $H_k$ and $H_{k'}$ are unbiased. 
\end{proof}

\section{$8$ class association schemes}
Linked systems of symmetric designs with specific parameters have the extended $Q$-bipartite double which yields an association scheme of mutually unbiased bases \cite[Theorem 3.6]{mmw}. 
Next we show an association scheme from our association schemes of class $5$ has a double cover and show a relation to an association scheme of class $4$ as a fusion scheme. 
\begin{theorem}
Let $A_0,A_1,\ldots,A_5$ be the adjacency matrices of the association scheme in Theorem~\ref{mn5}. 
Define \begin{align*}
\tilde{A}_0&=\begin{pmatrix} A_0&0\\0&A_0\end{pmatrix},
\tilde{A}_1=\begin{pmatrix} A_1&0\\0&A_1\end{pmatrix},
\tilde{A}_2=\begin{pmatrix} 0&A_1\\A_1&0\end{pmatrix},
\tilde{A}_3=\begin{pmatrix} A_2&A_2\\A_2&A_2\end{pmatrix},\\
\tilde{A}_4&=\begin{pmatrix} A_3&0\\0&A_3\end{pmatrix},
\tilde{A}_5=\begin{pmatrix} 0&A_3\\A_3&0\end{pmatrix},
\tilde{A}_6=\begin{pmatrix} A_4&A_5\\A_5&A_4\end{pmatrix},
\tilde{A}_7=\begin{pmatrix} A_5&A_4\\A_4&A_5\end{pmatrix},\\
\tilde{A}_8&=\begin{pmatrix} 0&A_0\\A_0&0\end{pmatrix}.
\end{align*}  
Then $\tilde{A}_0,\ldots,\tilde{A}_8$ form an association scheme.
\end{theorem}
\begin{proof}
Follows from the calculation in Theorem~\ref{mn5}.  
\end{proof}

\begin{remark}
\begin{enumerate}
\item The association scheme of class $8$ is also uniform.  
The second eigenmatrix and $B_8^*$ are as follows:
\begin{align*}
Q&=\left(\begin{smallmatrix}
1&2n(2n-1)&2n&2n-1&2n(2n-1)(m+1)&(2n-1)m&2nm&2n(2n-1)m&m\\
1&-2n&2n&2n-1&-2n(m+1)&(2n-1)m&2nm&-2nm&m\\
1&2n&-2n&2n-1&-2n(m+1)&(2n-1)m&-2nm&2nm&m\\
1&0&0&-1&0&-m&0&0&m\\
1&0&2n&2n-1&0&-2n+1&-2n&0&-1\\
1&0&-2n&2n-1&0&-2n+1&2n&0&-1\\
1&2n&0&-1&0&1&0&-2n&-1\\
1&-2n&0&-1&0&1&0&2n&-1\\
1&-2n(2n-1)&-2n&2n-1&2n(2n-1)(m+1)&(2n-1)m&-2nm&-2n(2n-1)m&m
\end{smallmatrix}\right),\displaybreak[0]\\
B_7^*&=\begin{pmatrix}
0&0&0&0&0&0&0&0&1\\
0&0&0&0&0&0&0&1&0\\
0&0&0&0&0&0&1&0&0\\
0&0&0&0&0&1&0&0&0\\
0&0&0&0&m&0&0&0&0\\
0&0&0&m&0&m-1&0&0&0\\
0&0&m&0&0&0&m-1&0&0\\
0&m&0&0&0&0&0&m-1&0\\
m&0&0&0&0&0&0&0&m-1
\end{pmatrix}.
\end{align*}
Thus the association scheme certainly satisfies \cite[Proposition 4.7]{dmm}. 
\item 
Letting $\tilde{B}_1=\tilde{A}_1+\tilde{A}_2+\tilde{A}_3$, $\tilde{B}_2=\tilde{A}_4+\tilde{A}_6$ and  $\tilde{B}_3=\tilde{A}_5+\tilde{A}_7$, $\tilde{B}_4=\tilde{A}_8$, 
we obtain a fusion association scheme of class $4$.  
The second eigenmatrix of the class $4$ fusion association scheme is as follows:
\begin{align*}
Q&=\begin{pmatrix}
1&4n^2&(4n^2-1)(m+1)&4n^2m&m\\
1&0&-m-1&0&m\\
1&2n&-0&-2n&-1\\
1&-2n&0&2n&-1\\
1&-4n^2&(4n^2-1)(m+1)&-4n^2m&m
 \end{pmatrix}.
\end{align*}
This association scheme is a $Q$-antipodal and $Q$-bipartite $Q$-polynomial scheme of class $4$, see \cite{lmo}.  
\end{enumerate}
\end{remark}


\begin{thebibliography}{99}
\bibitem{bi}
E. Bannai, T. Ito, Algebraic Combinatorics I: Association Schemes, Benjamin/Cummings, Menro Park, CA, 1984.

\bibitem{bk}
Darcy Best, Hadi Kharaghani, Unbiased complex Hadamard matrices and bases,
Cryptography and Communications - Discrete Structures, 
Boolean Functions and Sequences, {\bf 2} (2010), 199--209.

\bibitem{d}
E.  van. Dam, Three-class association schemes, 
{\it J. Algebraic.\ Combin.\ } 10(1) (1999), 69--107.

\bibitem{dmm}
E. van Dam, W. Martin, M. Muzychuk, 
Uniformity in association schemes and coherent configurations: cometric Q-antipodal schemes and linked systems,
{\it J. Combin. Theory Ser. A} {\bf 120} (2013), no. 7, 1401-1439. 

\bibitem{dd}
A. Deza, M. Deza, 
The ridge graph of the metric polytope and some relatives.  Polytopes: abstract, convex and computational (Scarborough, ON, 1993), 359-372. 

\bibitem{efhhh}
M. Erickson, S. Fernando, W. H. Haemers, D. Hardy, J. Hemmeter, 
Deza graphs: a generalization of strongly regular graphs. 
{\it J. Combin. Des.\ } 7 (1999), no. 6, 395-405. 

\bibitem{h}
D. G. Higman, 
Strongly regular designs of the second kind. 
{\it European J. Combin.\ } 16 (1995), no. 5, 479--490. 

\bibitem{hko}
W. H. Holzmann, H. Kharaghani, W. Orrick, 
On the real unbiased Hadamard matrices.
Combinatorics and graphs, 243--250, Contemp. Math., 531, Amer. Math. Soc.,
Providence, RI, 2010.

\bibitem{lmo}
N. LeCompte, W. J. Martin and W. Owens, On the equivalence between real mutually unbiased bases and a certain class of association schemes. 
{\it European J. Combin.\ } 31 (2010), no. 6, 1499--1512.

\bibitem{mmw}
W. J. Martin, M. Muzychuk, J. Williford, 
Imprimitive cometric association schemes: constructions and analysis. 
{\it J. Algebraic Combin.\ } 25 (2007), no. 4, 399–415.

\bibitem{ys}
J. Seberry and M. Yamada, {\em Hadamard matrices, sequences, and
block designs}, in Contemporary Design Theory: A Collection of
Surveys, J. H. Dinitz and D. R. Stinson, eds., John Wiley \& Sons,
Inc., 1992, pp. 431--560.
\end{thebibliography}
\end{document}